\documentclass[12pt]{article}
\usepackage{mathrsfs}
\usepackage{amsfonts}
\usepackage[all]{xy}
\usepackage{amsmath,amssymb}
\usepackage{amscd}

\baselineskip 5.5pt \lineskip 5.5pt \numberwithin{equation}{section}

\def\qed{\hfill$\Box$\par}

\def\qed{\ \ \ifhmode\unskip\nobreak\fi\ifmmode\ifinner
         \else\hskip5pt\fi\fi
 \hbox{\hskip5pt\vrule width4pt height6pt depth1.5pt\hskip 1 pt}}

\def\cl{\centerline}

\newtheorem{theo}{Theorem}[section]
\newtheorem{lemm}[theo]{Lemma}
\newtheorem{rema}[theo]{Remark}
\newtheorem{defi}[theo]{Definition}
\newtheorem{coro}[theo]{Corollary}
\newtheorem{prop}[theo]{Proposition}

\begin{document}
\cl{{\large\bf Whittaker Modules for the Schr\"{o}dinger Algebra}
\noindent\footnote{\footnotesize Supported partially by the National
Natural Science Foundation of China (No. 11271165, 11047030,
11171055) and the Youth Foundation of National Natural Science
Foundation of China (No.11101350).}} \vspace{16pt}
 \cl{Xiufu Zhang}
 \cl{\small School of Mathematics and Statistics}
 \cl{\small Jiangsu
Normal University, Xuzhou 221116, China}
\cl{Email:
xfzhang@jsnu.edu.cn}
 \vspace{16pt}
 \cl{Yongsheng Cheng}
 \cl{\small School of Mathematics and Information
Science }\cl{ \small Institute of Contemporary Mathematics, }
\cl{\small Henan University, Kaifeng 475004, China}

\numberwithin{equation}{section}

\begin{abstract}
In this paper,  the property and the classification the simple
Whittaker modules  for the schr\"{o}dinger algebra are studied. A
quasi-central element plays an important role in the study of
Whittaker modules of level zero. For the Whittaker modules of
nonzero level, our arguments use the Casimir operator of semisimple
Lie algebra $sl_2$ and the description of simple modules over
conformal Galilei algebras by R. L\"{u} , V. Mazorchuk and K. Zhao.
\vspace{2mm}\vspace{2mm}
\end{abstract}

\vskip 3mm\noindent{\section{Introduction}}

We denote by $\mathbb{Z}, \mathbb{Z}_+, \mathbb{N}$ and $\mathbb{C}$
the set of all integers, nonnegative integers, positive integers and
complex numbers, respectively.

The Schr\"{o}dinger group is the symmetry group of the free particle
Schr\"{o}dinger equation. The Lie algebra $\mathfrak{S}$ of this
group in the case of $(1+1)-$dimensional space-time is called the
Schr\"{o}dinger algebra which plays an important role in physics
applications (see [1-4]). Concretely, the Schr\"{o}dinger algebra
has a basis $\{e, h, f,p, q, z\},$ subject to the Lie brackets
\begin{eqnarray}
\begin{split}
&[h,e]=2e, &[&h,f]=-2f, &[&e,f]=h,\\
&[h,p]=p, &[&h,q]=-q, &[&p,q]=z,\\
&[e,q]=p, &[&p,f]=-q, &[&f,q]=0,\\
&[e,p]=0, &[&z,\mathfrak{S}]=0. &
\end{split}
\end{eqnarray}
 From the definition, we see that the
Schr\"{o}dinger algebra $\mathfrak{S}$ can be viewed as a semidirect
product $ \mathfrak{S}=\mathfrak{H}\rtimes sl_2 $ of two
subalgebras: a Heisenberg subalgebra
$\mathfrak{H}=\mathrm{span}\{p,q,z\}$ and
$sl_2=\mathrm{span}\{e,h,f\}.$ We also see that
$\mathfrak{S}=\bigoplus_{i=-2}^2\mathfrak{S}_i$ is
$\mathbb{Z}-$graded, where
$$\mathfrak{S}_{-2}=\mathbb{C}f,\mathfrak{S}_{-1}=\mathbb{C}q,
\mathfrak{S}_{0}=\mathbb{C}h+\mathbb{C}z,
\mathfrak{S}_{1}=\mathbb{C}p, \mathfrak{S}_{2}=\mathbb{C}e.$$
Obviously,
\begin{eqnarray}
\mathfrak{S}=\mathfrak{n}^{+}\oplus\mathfrak{h}\oplus\mathfrak{n}^{-}
\end{eqnarray}
is the Cartan decomposition according to the Cartan subalgebra
$\mathfrak{h}=\mathbb{C}h+\mathbb{C}z,$ where
$\mathfrak{n}^{-}=\mathrm{span}\{f,q\},
\mathfrak{n}^{+}=\mathrm{span}\{p,e\}.$

A module $V$ for Schr\"{o}dinger algebra $\mathfrak{S}$ is called a
weight module if it is the sum of all its weight spaces
$V_{\lambda}=\{v\in V|hv=\lambda v\}$ for some $\lambda\in
\mathbb{C}$. Vectors in $V_{\lambda}$ are called weight vectors. A
weight $\mathfrak{S}-$module $V$ is called a Harish-Chandra module
if all weight spaces are finite-dimensional.  A nonzero weight
vector $v\in V$ is called a highest weight vector if
 $\mathfrak{n}^{+}v=0$. $V$ is called a highest
weight module if it is generated by a highest weight vector $v$.

 Recently, the weight modules for
the Schr\"{o}dinger algebra are deeply studied.  In [5], the
irreducible lowest weight modules are classified by using the
technique of singular vectors. In [6], it is proved that all the
weight spaces of a simple weight module  which is neither a highest
weight module nor a lowest weight module have the same dimension. By
using the results in [5] and [6] and $\mathrm{Mathieu}^{,}$s
twisting functors (see [7]),  the author proved that the
Harish-Chandra modules can be twisted from a class of irreducible
highest weight modules in [8]. Thus the classification of simple
weight modules with finite-dimensional weight spaces are completely
obtained: such module is either a dense $sl_2-$module (see [9]) or a
highest (lowest) weight module or a module ``twisted" from a highest
weight module. Moreover, The classification of simple weight modules
over the Schr\"odinger algebra are completed in [10].

The first series of nonweight simple modules for complex semisimple
Lie algebras were constructed by B. Kostant in [11]. These modules
were called Whittaker modules because of their connection to
Whittaker equation in number theory. D. Arnal and G. Pinczon studied
in [12] a generalization of Whittaker modules.
  Since then the study of property
 of  Whittaker modules over various
algebras are open. The prominent role played by Whittaker modules is
illustrated by the main result in [13] about the
 classification of all simple modules for the Lie algebra
$sl_2:$  the simple $sl_2-$modules fall into three families: highest
(lowest) weight modules, Whittaker modules and a third family
obtained by localization.  Whittaker modules have many important
properties and are studied for many algebras, such as Virasoro
algebra, Heisenberg algebras, affine Lie algebras,  generalized Weyl
algebras, quantum enveloping algebra $U(sl_2)$, Schr\"{o}dinger-witt
algebra, generalized Virasoro algebra, twisted Heisenberg-Virasoro
algebra as well as a class of algebras similar to $U(sl_2)$ etc.
(see Refs. [14-23]). Moreover, quantum deformation of Whittaker
modules and modules induced from Whittaker modules are studied(see
Refs.[24],[25]).

In this paper, we classify the simple Whittaker modules over the
Schr\"{o}dinger algebra.

In section 2, we give some notations, identities and some lemmas
which will be used in the following sections. If $V$ is a simple
module over the Schr\"{o}dinger algebra then $z$ acts as a scalar
$\dot{z}$, called the level.

In section 3, we study the Whittaker modules of zero level. For
simplicity, we first  study  the Whittaker modules over a subalgebra
generated by $e,p,h,q,z$  in subsection 3.1. In subsection 3.2,
the classification of simple Whittaker modules of zero level is
given. A quasi-central element plays an important role in this part.

In section 4, we use the Casimir operator to study the simplicity of
Whittaker modules for semisimple Lie algebra $sl_2$ first. Then by
using the simple Whittaker modules for Heisenberg algebra
$\mathfrak{H}$ and $sl_2,$ we give the classification of simple
Whittaker modules of nonzero level for Schr\"{o}dinger algebra.

\noindent{\section{Definitions, identities and some lemmas}}

In this section, we first give some fundamental definitions and some
identities which will be used frequently in this paper.

\begin{defi} Let $\phi:\mathfrak{n}^{+}\rightarrow\mathbb{C}$ be any
nonzero Lie algebra homomorphism. Let $V$ be a
$\mathfrak{S}$-module.

 $\mathrm{(i)}$ A nonzero vector $v\in V$ is
called a Whittaker vector of type $\phi$ if $xv=\phi(x)v$ for all
$x\in\mathfrak{n}^{+}.$

$\mathrm{(ii)}$ $V$ is called a Whittaker module for $\mathfrak{S}$
of type $\phi$ if $V$ contains a cyclic Whittaker vector $v$ of type
$\phi$.
\end{defi}

\begin{defi}  Let
$\phi:\mathfrak{n}^{+}\rightarrow\mathbb{C}$ be any nonzero Lie
algebra homomorphism. Define a one-dimensional
$\mathfrak{n}^{+}$-module $\mathbb{C}_{\phi}=\mathbb{C}w$ by
$ew=\phi(e)w, pw=\phi(p)w.$ The induced module
$$W_{\phi}=U(\mathfrak{S})\otimes_{U(\mathfrak{n}^{+})}\mathbb{C}_{\phi}$$
is called the universal Whittaker module of type $\phi$ in the sense
that for any Whittaker module $V$ with whittaker vector $v$ of type
$\phi,$ there is a unique surjective homomorphism
$f:W_{\phi}\rightarrow V$ such that $uw\mapsto uv, \forall u\in
U(\mathfrak{b}^{-}).$
\end{defi}

For simplicity,  we always denote $\phi(x)$ by $\dot{x}$ for $x=e,p$
in the following.

We give some useful identities in $U(\mathfrak{S})$:
\begin{lemm}In $U(\mathfrak{S}),$ the following identities hold:
\begin{align}
&[p,f^n]=-nqf^{n-1},\\
&[p,h^n]=\sum_{i=1}^{n}(-1)^i\binom{n}{i}
h^{n-i}p,\\
&[q,h^n]=\sum_{i=1}^{n}\binom{n}{i}
h^{n-i}q,\\
&[e,q^n]=nq^{n-1}p+\frac{n(n-1)}{2}q^{n-2}z,\\
&[e,f^n]=nf^{n-1}h-n(n-1)f^{n-1},\\
&[e,h^n]=\sum_{i=1}^{n}(-2)^i\binom{n}{i} h^{n-i}e.
\end{align}
\end{lemm}
\begin{proof} We only prove (2.6), the other identities are obvious or can be
proved similarly. Let $R_h$ (resp. $L_h$) be the right (resp. left)
multiplication operator by $h.$ Then
\begin{align*}
& [e,h^n]=R_h^n(e)+h^ne=(L_h-\mathrm{ad}h)^ne+h^ne.
\end{align*}
Using the binomial formula since $L_h$ and $\mathrm{ad}h$ are
commuting, we obtain (2.6).
\end{proof}

\begin{lemm} For an arbitrary homomorphism of Lie algebra $\phi:\mathfrak{n}^+\rightarrow\mathbb{C},$
if  $W_{\phi}$ is the universal Whittaker module type $\phi$, then
any Whittaker vectors of $W_{\phi}$ are all of type $\phi.$
\end{lemm}
\begin{proof}
Suppose $\phi^{'}:\mathfrak{n}^{+}\rightarrow \mathbb{C}$ is a Lie
algebra homomorphism which is different from $\phi$ and $v^{'}=\sum
a_{ijk}f^iq^jh^kw$ is a Whittaker vector of type $\phi^{'}.$ Define
$A:=\mathrm{max}\{i+j+k|a_{ijk}\neq0\}$ and let $\succ$ be the
lexicographic order on $\mathbb{N}\times\mathbb{N}\times\mathbb{N}.$
Suppose $(\bar{i},\bar{j},\bar{k})$ is the maximal element of
$\{(i,j,k)|i+j+k=A\}$ and $a_{\bar{i}\bar{j}\bar{k}}=1.$ Then
\begin{align}
v=f^{\bar{i}}q^{\bar{j}}h^{\bar{k}}w+\sum_{i+j+k<A\ or\
(i,j,k)\prec(\bar{i},\bar{j},\bar{k})}a_{ijk}f^iq^jh^kw.
\end{align}
If $\phi^{'}(x)\neq \phi(x),$ $x=e$ or $p,$ on one hand, we have
\begin{align}
xv=\phi^{'}(x)v.
\end{align}
On the other hand, we have
\begin{align}
xv=\phi(x)v+([x,f^{\bar{i}}q^{\bar{j}}h^{\bar{k}}]+\sum_{i+j+k<A\
or\ (i,j,k)\prec(\bar{i},\bar{j},\bar{k})}a_{ijk}[x,f^iq^jh^k])w.
\end{align}
By Lemma 2.3 and identity (2.9), we see that
\begin{align}
xv=\phi(x)f^{\bar{i}}q^{\bar{j}}h^{\bar{k}}w+\sum_{i+j+k<A\ or\
(i,j,k)\prec(\bar{i},\bar{j},\bar{k})}b_{ijk}f^iq^jh^kw,
\end{align} where  $b_{ijk}\in\mathbb{C}.$
Then the lemma follows from (2.7), (2.8) and (2.10).
\end{proof}

\begin{lemm} Suppose $\mathfrak{L}$ be a
finite-dimensional Lie algebra with the triangular decomposition
$\mathfrak{L}=\mathfrak{L}^{+}\oplus \mathfrak{L}_0\oplus
\mathfrak{L}^{-}$ according to the Cartan subalgebra
$\mathfrak{L}_0$, $\phi:\mathfrak{L}^{+}\rightarrow \mathbb{C}$ is a
Lie algebra homomorphism.

$\mathrm{(i)}$ If $V$ is a Whittaker $\mathfrak{L}-$module with
cyclic Whittaker vector of type $\phi$,  then every submodule $W$ of
$\mathfrak{L}-$module $V$ contains a Whittaker vector of type
$\phi.$

$\mathrm{(ii)}$ If the vector space of Whittaker vectors of type
$\phi$ is one dimensional, then $V$ is irreducible.
\end{lemm}
\begin{proof} We see that $\{x-\phi(x)|x\in\mathfrak{L}^+\}$ act locally nilpotent on $V$ since $\mathrm{ad}x$
acts locally nilpotent on $U(\mathfrak{L})$ for any
$x\in\mathfrak{L}^+.$ Thus  for any nonzero submodule $W$ of $V$ and
any nozero vector $v\in W,$ $U(\mathfrak{L}^+)v$ is a
finite-dimensional $\mathfrak{L}^+$-submodule. By Lie Theorem, (i)
holds. (ii) is a direct corollary of (i).
\end{proof}

The following result is a direct corollary of Lemma 2.5.
\begin{coro}
Let $V$ be any Whittaker module of type $\phi$ for Schr\"{o}dinger
algebra $\mathfrak{S}.$ Then every submodule $W$ of $V$ contains a
Whittaker vector of type $\phi.$
\end{coro}

\vskip 3mm\noindent{\section{Whittaker modules of zero level}}

In this section, we classify the simple Whittaker modules of zero
level for the Schr\"{o}dinger algebra $\mathfrak{S}.$ For
simplicity, we first study the Whittaker vectors and Whittaker
modules for the subalgebra $\mathfrak{S}_1$ generated by $\langle
p,q,e,h\rangle$ in subsection 3.1. Then in subsection 3.2, we study
the Whittaker $\mathfrak{S}$-modules induced from the simple
Whittaker $S_1-$modules, and determine the classification of simple
Whittaker $\mathfrak{S}$-modules of zero level. In our study, a
quasi-central element (see Lemma 3.5) of $U(\mathfrak{S})$ given in
[26] plays an important role.

\vskip 3mm\noindent{\subsection{Whittaker modules of zero  level for
subalgebra $\mathfrak{S_1}$ }}

We denote the subalgebra of Schr\"{o}dinger algebra $\mathfrak{S}$
generated by $\{e,p,h,q\}$ by $\mathfrak{S}_1$. The universal
Whittaker module of type $\phi$ for $\mathfrak{S}_1$ is denoted by $
W_{\phi}.$ Then for any $v=\in W_{\phi},$ $v$ can be written in the
form:
\begin{equation} \label{eq:1}
v=\sum_{i=0}^na_ih^i\psi_i(q)w,
\end{equation}
where $\psi_i(x)\in\mathbb{C}[x],\psi_n(x)\neq0$ and $a_n\neq0$.

\begin{prop}
$\mathrm{(i)}$  Suppose $\dot{e}\neq0,\dot{p}=0$ and $W_{\phi}$ is
the universal Whittaker module for $\mathfrak{S}_1$ with cyclic
Whittaker vector $w$. Then $v\in W_{\phi}$ is a Whittaker vector if
and only if $v=uw$ for some $u\in\mathbb{C}[q].$

$\mathrm{(ii)}$ If $\dot{p}\neq0,$ the set of Whittaker vectors in
$W_{\phi}$ is $\mathbb{C}w$.
\end{prop}
\begin{proof}
$\mathrm{(i)}$  By (2.4) and $[p,q]w=0,$ it is obvious that $uw$ is
a Whittaker vector if $u\in\mathbb{C}[q]$. For any Whittaker vector
$v\in W_{\phi}$ with form (3.1), if $n\neq0,$ by (2.6), we have
\begin{equation} \label{eq:1}
(e-\dot{e})^{n}v=a_{n}(-2)^{\frac{n(n+1)}{2}}n!\dot{e}\psi_{n}(q)w\neq0,
\end{equation}
a contradiction.

$\mathrm{(ii)}$  For any nonzero vector $v\in
W_{\phi}\setminus\mathbb{C}w,$ we can assume $
v=\sum_{i=0}^na_ih^i\psi_i(q)w, $ where $a_i\in\mathbb{C}, a_n\neq0,
\psi_i(x)\in\mathbb{C}[x]$ and $n+\mathrm{deg}(\psi_i(x))>0.$

If $n=0, \mathrm{deg}(\psi_i(x))>0,$ we see that
$v=\sum_{j=0}^mb_jq^jw,$ where $m>0,b_j\in\mathbb{C},b_m\neq0.$ By
(2.4), we can easily get that
$$(e-\dot{e})v=\dot{p}\sum_{j=1}^mjb_jq^{j-1}w\neq0,$$ which means
$v$ is not a Whittaker vector.

If  $n\neq0,$ by (2.2) we see that
$$(p-\dot{p})v=-n\dot{p}a_nh^{n-1}\psi_n(q)w+\sum_{k<n-1}c_kh^k\psi_k^{'}(q)w\neq0,$$where
$c_k\in\mathbb{C},\psi_k^{'}(x)\in \mathbb{C}[x].$ Thus $v$ is not a
Whittaker vector.
\end{proof}

\begin{lemm}  If $\dot{e}\neq0,\dot{p}=0$, $M$ is a
$\mathfrak{S}_1$-submodule of $W_{\phi},$
 then
$M$ is a maximal submodule of $W_{\phi}$ if and only if there exists
an a scalar $\xi$ such that $M$ is generated by $(q-\xi)w.$
\end{lemm}
\begin{proof}
Suppose there exists a scalar $\xi$ such that $M$ is generated by
$(q-\xi)w.$ For arbitrary vector $v\in M$, $v$ is of the form
$$v=\sum_{i,j}a_{i,j}h^iq^j(q-\xi)w.$$ By
$(e-\dot{e})(q-\xi)w=0=p(q-\xi)w$ and identities (2.2), (2.3) and
(2.6) we see that $w\notin M.$ Thus $M$ is a proper submodule of
$W_{\phi}.$ For any $v^{'}\notin M,$ we have $v^{'}\equiv
\sum_{i}a_ih^iw(\mathrm{mod}M).$ Thus we can assume that
$$v^{'}=h^kw+\sum_{j<k}a_jh^jw\in (U(\mathfrak{S}_1)w+M).$$
Then by using (2.6) we can easily deduce that $w\in
(U(\mathfrak{S}_1)w+M),$ which means
$U(\mathfrak{S}_1)w+M=W_{\phi}.$ Thus $M$ is maximal.

Now suppose $M$ is a maximal $\mathfrak{S}_1-$submodule of
$W_{\phi}$ and $v\in M$ is an arbitrary nonzero vector. Then $v$
must be of the form (3.1).  We choose $v\in M$ such that $n$ be the
minimal. If $n\neq0,$ by (3.2), we see that $\psi_{n}(q)w\in M$,
which is contrary to the minimal of $n.$ Thus there exists
polynomial $0\neq\psi_0(x)\in\mathbb{C}[x]$ such that $\psi_0(q)w\in
M.$ Set
$$\underline{k}=\mathrm{min}\{\mathrm{deg}(\psi(x))|\psi(q)w\in M\}$$ and
choose monic polynomial $\underline{\psi}(x)\in\mathbb{C}[x]$ such
that $\mathrm{deg}(\underline{\psi}(x))=\underline{k}$ and
$\underline{\psi}(q)w\in M$. We see that $\underline{k}>0$ since
$w\notin M.$ Now we claim that $M$ is generated by
$\underline{\psi}(q)w.$ In fact, for any $0\neq v\in M$ with form
(3.1), similar as we have done in (3.2), we find that $\psi_i(q)w\in
M$ for all $i$ with $a_i\neq0.$ Thus $\underline{\psi}(x)|\psi_i(x)$
by the definition of $\underline{\psi}(x)$ and division algorithm in
$\mathbb{C}[x].$ So $M=U(\mathfrak{S}_1)\underline{\psi}(q)w.$
Furthermore, we see that $\underline{\psi}(x)$ is irreducible since
$M$ is maximal, which means there exists a $\xi\in\mathbb{C}$ such
that $\underline{\psi}(x)=x-\xi$ since
$\underline{\psi}(x)\in\mathbb{C}$ is monic.
\end{proof}

By Lemma 3.2 we see that if
$\phi(e)=\dot{e}\neq0,\phi(p)=\dot{p}=0,$ $M_{\xi}$ is a maximal
$\mathfrak{S}_1-$submodule of $W_{\phi}$ generated by $(q-\xi)w,$
then the corresponding simple quotient module
$W_{\phi}/M_{\xi}=\mathbb{C}[h]\bar{w}$ as vector space. We will
denote the simple quotient module $W_{\phi}/M_{\xi}$ by
$\mathbb{C}[h]_{\phi}^{\xi}\bar{w},$ where $\bar{w}=w+M_{\xi}.$

\begin{prop}$\mathrm{(i)}$  If $\dot{e}\neq0,\dot{p}=0,$
  any simple Whittaker $\mathfrak{S}_1-$module of
type $\phi$ is isomorphic to $\mathbb{C}[h]_{\phi}^{\xi}\bar{w}$ for
some $\xi\in\mathbb{C}.$

$\mathrm{(ii)}$ If $\dot{p}\neq0,$ then $W_{\phi}$ is simple
Whittaker $\mathfrak{S}_1-$module.
\end{prop}
\begin{proof}
$\mathrm{(i)}$ For any simple Whittaker $\mathfrak{S}_1-$module $V$,
there exists a maximal submodule $M$ of $W_{\phi}$ such that $V\cong
W_{\phi}/M.$ By Lemma 3.2, there exists some $\xi\in\mathbb{C}$ such
that $M=M_{\xi},$ so we have $V\cong
\mathbb{C}[h]_{\phi}^{\xi}\bar{w}$ for some $\xi\in\mathbb{C}.$

$\mathrm{(ii)}$ This is the direct corollary of Lemma 2.5 and
Proposition 3.1 (ii).
\end{proof}

\vskip 3mm\noindent{\subsection{Simple Whittaker modules for
$\mathfrak{S}$ of zero level }}

In subsection 3.1, we obtain all the simple $\mathfrak{S}_1-$modules
$\mathbb{C}[h]_{\phi}^{\xi}$ (when $\dot{e}\neq0,\dot{p}=0$) and
$W_{\phi}$ (when $\dot{p}\neq0)$. For $\dot{e}\neq0,\dot{p}=0,\xi\in
\mathbb{C},$ in the following, we set
 $$L_{\phi}^{\xi}:=U(\mathfrak{S})\otimes_{U(\mathfrak{S}_1)}\mathbb{C}[h]_{\phi}^{\xi}.$$
For $\dot{p}\neq0,$ we set
$$\widetilde{W}_{\phi}=U(\mathfrak{S})\otimes_{U(\mathfrak{S}_1)}V.$$

In this subsection, we first study the Whittaker vectors and
simplicity of the above induced Whittaker modules. Then we classify
the simple Whittaker modules of zero level over $\mathfrak{S}$.

\begin{lemm} Suppose $\dot{p}\neq0.$ For $\bar{\mathfrak{c}}\in U(\mathfrak{S})$ with form $\bar{\mathfrak{c}}=f+\bar{\mathfrak{c}},$ where
$\bar{\mathfrak{c}}_1\in U(\mathfrak{S}_1),$ such that
$\bar{\mathfrak{c}}w$ is a Whittaker vector, then
$$\bar{\mathfrak{c}}=f-\frac{1}{\dot{p}}q(1+h)-\frac{\dot{e}}{\dot{p}^2}q^2+c,$$
where $c\in\mathbb{C}.$
\end{lemm}
\begin{proof}
Suppose $\bar{\mathfrak{c}}=f+\sum_{i=0}^na_i\psi_i(q)h^i,$ where
$a_i\in\mathbb{C},a_{n}\neq0,\psi_i(x)\in\mathbb{C}[x].$ Then by
$(p-\dot{p})\bar{\mathfrak{c}}w=0,$ we get that
$$-qw+\dot{p}\sum_{i=1}^na_i\psi_i(q)\big(\sum_{k=1}^i(-1)^k\binom{i}{k}h^{i-k}\big)w=0.$$
So $n=1$  and $-qw-\dot{p}a_1\psi_1(q)w=0.$ Thus
$\psi_1(q)=-\frac{1}{a_1\dot{p}}q.$ If we assume that
$\psi_1(x)=b_0+b_1x+\cdots+b_sx^s,$ then we can rewrite
$\bar{\mathfrak{c}}$ as following
$$\bar{\mathfrak{c}}=f+a_0(b_0+b_1q+\cdots+b_sq^s)-\frac{1}{\dot{p}}qh.$$
By $(e-\dot{e})\bar{\mathfrak{c}}w=0,$ we get that $s=2$ and
$a_0b_1\dot{p}+2a_0b_2\dot{p}q+1+\frac{2\dot{e}}{\dot{p}}q=0.$ So
$$b_1=-\frac{1}{a_0\dot{p}}, b_2=-\frac{\dot{e}}{a_0\dot{p}^2},$$
which means that
$\bar{\mathfrak{c}}=(f-\frac{1}{\dot{p}}q(1+h)-\frac{\dot{e}}{\dot{p}^2}q^2)+a_0b_0.$
\end{proof}

From Lemma 3.4, we see that
$$(\bar{\mathfrak{c}}-a)w=(f-\frac{1}{\dot{p}}q(1+h)-\frac{\dot{e}}{\dot{p}^2}q^2)w=(fp^2-q(1+h)p-q^2e)(\frac{1}{\dot{p}^2}w).$$

The following key lemma is from [10], we will call $\mathfrak{c}$ a
quasi-central element of $U(\mathfrak{S}):$
\begin{lemm} Let $\mathfrak{c}=fp^2-q(1+h)p-q^2e.$ Then the center of
$U(\mathfrak{S}/\mathbb{C}z)=\mathbb{C}[\mathfrak{c}].$
\end{lemm}

We denote by $U(\langle q,h\rangle)$ the universal enveloping
algebra of the subalgebra generated by $q$ and $h.$ For
$u=\sum_{i=0}^na_i\psi_{i}(q,h)f^i\in U(\mathfrak{S}),$ where
$\psi_{i}(q,h)\in U(\langle q,h\rangle)$,  $a_i\in\mathbb{C},$
$a_n\neq0,$ $\psi_{n}(q,h)\neq0,$ we define $\mathrm{deg}_f(u)=n.$
Define
$$W_{\phi}^k:=\mathrm{span}\{uw|\mathrm{deg}_f(u)\leq k.\}$$

\begin{lemm}
If $\dot{p}\neq0,$ then for any $a\in \mathbb{C}, i\in\mathbb{N},$
there exists $\psi_k(q,h)\in U(\langle q,h\rangle),$ $
k=0,\cdots,i,$ such that
$$f^iw=\frac{1}{\dot{p}^{2i}}(\mathfrak{c}-a)^iw+\psi_1(q,h)(\mathfrak{c}-a)^{i-1}w+\cdots+\psi_{i-1}(q,h)(\mathfrak{c}-a)w+\psi_i(q,h)w.$$
\end{lemm}
\begin{proof}  Suppose Lemma
3.6 is proved for $i\leq k.$  By Lemma 3.4 and Lemma 3.5, we have
\begin{align*}
&f^{k+1}w=
f^{{k}}(f-\frac{1}{\dot{p}}q(1+h)-\frac{\dot{e}}{\dot{p}^2}q^2)w(\mathrm{mod}W_{\phi}^k)
\\
&\equiv f^{{k}}(fp^2-q(1+h)p-q^2e)\frac{1}{\dot{p}^2}w
(\mathrm{mod}W_{\phi}^k)
\\
&\equiv f^{{k-1}}(fp^2-q(1+h)p-q^2e)f\frac{1}{\dot{p}^2}w
(\mathrm{mod}W_{\phi}^k)\\
&\equiv f^{{k-2}}(fp^2-q(1+h)p-q^2e)^2\frac{1}{\dot{p}^4}w
(\mathrm{mod}W_{\phi}^k)\\
&\equiv \cdots\\
&\equiv (fp^2-q(1+h)p-q^2e)^{k+1}\frac{1}{\dot{p}^{2(k+1)}}w
(\mathrm{mod}W_{\phi}^k).
\end{align*}
By induction, Lemma 3.6 holds.
\end{proof}

\begin{prop} $\mathrm{(i)}$ If $\dot{e}\neq0,  \dot{p}=0, \xi\neq0,$  the set of  Whittaker vectors of
$L_{\phi}^{\xi}$ is $\mathbb{C}\bar{w}.$

$\mathrm{(ii)}$ If $\dot{p}\neq0, $  the set of  Whittaker vectors
of $\widetilde{W}_{\phi}$ is $\mathbb{C}[\mathfrak{c}]w,$ where
\begin{align*}
\mathfrak{c}=fp^2-q(1+h)p-q^2e.
\end{align*}
\end{prop}
\begin{proof}
$\mathrm{(i)}$ Suppose $\dot{e}\neq0,\dot{p}=0,\xi\neq0.$ For any
nonzero $v\in L_{\phi}^{\xi},$ we can assume that
$$v=\sum_{k=0}^{n}a_kf^k\psi_k(h)\bar{w},$$ where
$a_n\neq0,\psi_k(x)\in \mathbb{C}[x],\psi_n(x)\neq0.$ Then if $v$ is
a Whittaker vector, by identities (2.1) and (2.2),  we have
$$pv=\sum_{k=0}^{n}a_k(-k)f^{k-1}q\psi_k(h)\bar{w}=0.$$ By (2.3) and the fact
that $q\bar{w}=\xi\bar{w}\neq0,$ we have $n=0.$ So $v$ can be
written by
$$v=\psi_0(h)\bar{w}=h^k\bar{w}+a_{k-1}h^{k-1}\bar{w}+\cdots +a_0h_0,$$
where $k\in\mathbb{Z}_+$ and $a_i\in\mathbb{C}, i=0,i,\cdots,k-1.$
Then by using $(e-\dot{e})v=0$ and (2.6) we deduce that $k=0$, which
means $v\in \mathbb{C}\bar{w}.$

$\mathrm{(ii)}$ By Lemma 3.5, we see that every nonzero vector of
$\mathbb{C}[\mathfrak{c}]w$ is a Whittaker vector. For any Whittaker
vector $v\in \widetilde{W}_{\phi}$, there exists $0\neq
u=\sum_{i=0}^na_i\psi_{i}(q,h)f^i\in U(\mathfrak{S})$ with
$\mathrm{deg}_f(u)=n$ such that $v=uw.$ If $n=0,$ by Proposition
3.1(ii), we see that $\psi_{0}(q,h)\in\mathbb{C}^*,$ so $v\in
\mathbb{C}w\subseteq\mathbb{C}[\mathfrak{c}]w.$ Suppose for $n\leq
k$ we have proved that $v\in \mathbb{C}[\mathfrak{c}]w.$ Now we
begin to investigate the case of $n=k+1.$
 Using $(p-\dot{p})$ and
$(e-\dot{e})$ to act in turn on $v$, we can easily deduce that
$\psi_{k+1}(q,h)\in\mathbb{C}^{*}.$ For simplicity, we can assume
that $a_{k+1}=1=\psi_{k+1}(q,h).$ So by Lemma 3.6, we see that
\begin{align*}
v&=f^{k+1}w+\sum_{i=0}^ka_i\psi_{i}(q,h)f^iw\in\mathbb{C}[\mathfrak{c}]w.
\end{align*}Then, by
induction, Proposition 3.7 holds.
\end{proof}
\begin{theo}
$\mathrm{(i)}$ If $\dot{e}\neq0,\dot{p}=0,\xi\neq0,$ the Whittaker
module
 $L_{\phi}^{\xi}$
 is simple.

 $\mathrm{(ii)}$ If $\dot{p}\neq0,$  for any
 $a\in\mathbb{C},$ the Whittaker module $\widetilde{W}_{\phi}$ has the following filtration:
 $$\widetilde{W}_{\phi}=W^0\supset W^1\supset W^2\supset\cdots \supset W^n\supset
 \cdots,$$ where $W^i$ is a Whittaker submodule defined by $W^i=U(\mathfrak{S})(\mathfrak{c}-a)^iw.$ More precisely, $\widetilde{W}_{\phi}/W^1$
  is simple and $W^i/W^{i+1}\cong \widetilde{W}_{\phi}/W^1$ for all
 $i\in\mathbb{N}.$
\end{theo}
\begin{proof}
$\mathrm{(i)}$ follows from Corollary 2.6 and Proposition 3.7.

$\mathrm{(ii)}$ Suppose $\dot{p}\neq0.$ By Lemma 3.6, we see that
for any $0\neq v\in \widetilde{W}_{\phi}\setminus W_1,$
 there exists $\psi_i(q,h)\in
U(\langle q,h\rangle), i=0,\cdots,n, \psi_0(q,h)\neq0,
\psi_n(q,h)\neq0$ such that
$$v=\psi_n(q,h)(\mathfrak{c}-a)^nw+\psi_{n-1}(q,h)(\mathfrak{c}-a)^{n-1}w+\cdots+\psi_{1}(q,h)(\mathfrak{c}-a)w+\psi_0(q,h)w.$$
Thus $\psi_0(q,h)w\in (U(\mathfrak{S})v+W_1).$ Then by Proposition
3.3 (ii), we have $w\in(U(\mathfrak{S})v+W_1)$ and
$U(\mathfrak{S})v+W_1=\widetilde{W}_{\phi},$ which means $W_1$ is a
maximal submodule of $\widetilde{W}_{\phi}.$

By Lemma 3.5 and Lemma 3.6, for any element $v\in W^i,$ $v$ is of
the form
$$v=\sum_{k=i}^n\psi_k(q,h)(\mathfrak{c}-a)^kw,$$ where $n\geq i, \psi_n(q,h)\neq0.$
Define linear map:
\begin{align*}
g:\ \ \ W^i\ \ \ &\rightarrow\ \ \ \widetilde{W}_{\phi}/W^1\\
\psi(q,h)(\mathfrak{c}-a)^{k+i}w & \mapsto
\psi(q,h)(\mathfrak{c}-a)^{k}\bar{w}.
\end{align*}
It is easy to see that $g$ is a module homomorphism and
$\mathrm{ker}(g)=W^{i+1}$, which means that $W^{i}/W^{i+1}\cong
\widetilde{W}_{\phi}/W^1.$
\end{proof}

In the following, we denote $\widetilde{W}_{\phi}/W^1$ by
$M_{\phi}^{a}.$ So $M_{\phi}^{a}$ is a simple  Whittaker module with
cyclic Whittaker vector $\bar{w}:=w+W^1.$ By Proposition 3.7 (ii)
and the definition of $M_{\phi}^{a}$, we have the following lemma
immediately:

\begin{lemm}
The set of Whittaker vectors of $M_{\phi}^{a}$ is
$\mathbb{C}\bar{w}.$
\end{lemm}

\begin{theo} Suppose $V$ is any simple Whittaker module of zero level  of type $\phi$ for
$\mathfrak{S}$.

$\mathrm{(i)}$ If $\dot{e}\neq0,\dot{p}=0,$ then $V$ is either a
simple $sl_2-$module of type $\phi$ or  isomorphic to
 $L_{\phi}^{\xi}$
 for some $\xi\neq0.$

 $\mathrm{(ii)}$ If $\dot{p}\neq0,$ then  $V$   is
isomorphic to $M_{\phi}^a$ for some $a\in\mathbb{C}.$
\end{theo}
\begin{proof}
If $V$ is a simple Whittaker module of with cyclic Whittaker vector
$w_v$ of type $\phi$, by the universal property of
$\widetilde{W}_{\phi},$ there exists a surjective module
homomorphism $g:\widetilde{W}_{\phi}\rightarrow V$ with
$uw\rightarrow uw_v.$

$\mathrm{(i)}$ If $\phi$ satisfies
$\phi(e)=\dot{e}\neq0,\phi(p)=\dot{p}=0,$ let $V_1\subseteq V$ be
the $\mathfrak{S}_1$-module generated by $w_v.$ Then as
$\mathfrak{S}_1$-module, $V$ has the following direct sum
decomposition:
$$V=V_1\oplus fV_1/V_1\oplus f^2V_1/fV_1\oplus\cdots\oplus f^nV_1/f^{n-1}V_1\oplus\cdots.$$
Since $V$ is simple as $\mathfrak{S}-$module, we see that $V_1$ is
simple as $\mathfrak{S}_1$-module. Then by Proposition 3.3, there
exists $\xi\in \mathbb{C}$ such that $V_1\cong
\mathbb{C}[h]_{\phi}^{\xi}\bar{w}.$

\noindent{\bf{Case (1)}} If $\xi=0,$ then $qw_v=0,$ combing with
$ew_v=zw_v=0,$ we see that $V$ is just a simple Whittaker
$sl_2-$module.

\noindent{\bf{Case (2)}} If $\xi\neq0,$ then as
$\mathfrak{S}-$module, $V=\mathbb{C}[f]V_1$ is a quotient module of
$L_{\phi}^{\xi}=U(\mathfrak{S})\otimes_{U(\mathfrak{S}_1)}\mathbb{C}[h]_{\phi}^{\xi}\bar{w}.$
Thus $V\cong L_{\phi}^{\xi}$ since $L_{\phi}^{\xi}$ is simple by
Theorem 3.8 (i).

$\mathrm{(ii)}$ If $\phi$ satisfies $\phi(p)=\dot{p}\neq0,$  by
Proposition 3.7 (ii),  for $n\in\mathbb{Z}_+$ and $x=e$ or $p$ we
have
$$x\mathfrak{c}^nw_v=x(\mathfrak{c}^ng(w))=g(x\mathfrak{c}^nw)=\dot{x}g(\mathfrak{c}^nw)=\dot{x}\mathfrak{c}^ng(w),$$
which means $\mathbb{C}[\mathfrak{c}]w_v$ is a subset of Whittaker
vector of $V$. Since $V$ is simple, the Whittaker vectors are of the
form $cw_v$ for some $c\in\mathbb{C},$ thus $\mathfrak{c} w_v=aw_v$
for some $a\in \mathbb{C}.$ So
$U(\mathfrak{S})(\mathfrak{c}-a)w_v=0,$ which means
$U(\mathfrak{S})(\mathfrak{c}-a)w\subseteq \mathrm{ker}(g).$ Thus
$V\cong \widetilde{W}_{\phi}/\mathrm{ker}(g)$ is a nonzero submodule
of $M_{\phi}^{a},$ thus $V\cong M_{\phi}^{a}$ since $ M_{\phi}^{a}$
is simple by Theorem 3.8.
\end{proof}

\begin{coro}
Let $V$ be a Whittaker $\mathfrak{S}-$module of level zero of type
$\phi:e \rightarrow \dot{e}, p \rightarrow \dot{p}$.

$\mathrm{(i)}$ If $\dot{e}\neq0,\dot{p}=0,$ then $V$ is simple if
and only if $q$ acts on cyclic Whittaker vector as a scalar.

 $\mathrm{(ii)}$ If $\dot{p}\neq0,$ then  $V$   is simple if and
 only if
$\mathfrak{c}$ acts as a zero.
\end{coro}

\begin{theo} $\mathrm{(i)}$ If
$\dot{e}\neq0,\dot{p}=0, \xi, \xi^{'}\neq0,$ then
$L_{\phi^{'}}^{\xi^{'}}\cong L_{\phi}^{\xi}$ if and only if
$\phi^{'}=\phi,$ and $\xi^{'}=\xi.$

$\mathrm{(ii)}$  If $\dot{p}\neq0,$ then $M_{\phi^{'}}^{a^{'}}\cong
M_{\phi}^{a}$ if and only if $\phi^{'}=\phi,$ and $a^{'}=a.$
\end{theo}
\begin{proof}
 $\mathrm{(i)}$ The sufficiency is trivial. By Proposition 3.7 (i) and Lemma 2.4, we see that
 the necessity holds.

 $\mathrm{(ii)}$ We need only to prove the necessity. By Lemma 3.9,
 we get
 $\phi^{'}=\phi$ immediately. Suppose $\bar{w}$ (resp.
 $\bar{w}^{'}$)
 is a cyclic Whittaker vector of $M_{\phi}^a$ (resp.
 $M_{\phi^{'}}^{a^{'}}$) and $g: M_{\phi}^a\rightarrow M_{\phi^{'}}^{a^{'}}$ is the module isomorphism such that $g(\bar{w})=g(w)$.
 Then we have
 $$0=(\mathfrak{c}-a^{'})\bar{w}^{'}=(\mathfrak{c}-a^{'})g(\bar{w})=g((\mathfrak{c}-a^{'})(\bar{w}))=(a-a^{'})\bar{w}^{'},$$
 that is $a^{'}=a.$
\end{proof}

\begin{rema}
By Theorem 3.10 and Theorem 3.11, the Whittaker modules of zero
level for Schr\"{o}dinger algebra are classified.
\end{rema}

\noindent{\section{Whittaker modules of nonzero level}}

In this section, we classify the simple Whittaker modules of nonzero
level. Our arguments use the  Casimir operator of semisimple Lie
algebra $sl_2$ and the description of simple modules over conformal
Galilei algebras in [24].

\begin{defi}
Let $\mathbb{C}_{\dot{p},\dot{z}}=\mathbb{C}w$ be a one-dimensional
vector space. By the action of
$$pw=\dot{p}w, zw=\dot{z}w,$$ $\mathbb{C}_{\dot{p},\dot{z}}$ can be viewed as a
$\mathbb{C}p\oplus\mathbb{C}z-$module. The induced module
$$M_{\mathfrak{H}}(\dot{p},\dot{z})=U(\mathfrak{H})\otimes_{U(\mathbb{C}p\oplus\mathbb{C}z)}\mathbb{C}_{\dot{p},\dot{z}}$$
is called a Whittaker module for Heisenberg Lie algebra
$\mathfrak{H}.$
\end{defi}

The following proposition is from [15], it can also be checked
easily:

\begin{prop}
For $\dot{z}\neq0,$ the following hold:

$\mathrm{(i)}$ $M_{\mathfrak{H}}(\dot{p},\dot{z})$ is the unique (up
to isomorphism) irreducible Whittaker $\mathfrak{H}-$module on which
$z$ acts by $\dot{z}\neq0$.

$\mathrm{(ii)}$ $M_{\mathfrak{H}}(\dot{p},\dot{z})\cong
M_{\mathfrak{H}}(\dot{p}^{'},\dot{z}^{'})$ if and only if
$\dot{p}=\dot{p}^{'}$ and $\dot{z}=\dot{z}^{'}$.
\end{prop}

\begin{defi}
Suppose $\mathbb{C}_{\dot{e}}=\mathbb{C}w$ be a $\mathbb{C}e-$module
by the action of $ew=\dot{e}w.$ The induced module
$$W_{sl_2}(\dot{e})=U(\mathfrak{H})\otimes_{U(\mathbb{C}e)}\mathbb{C}_{\dot{e}}$$
is called a Whittaker module for complex semisimple Lie algebra
$sl_2.$
\end{defi}

Let $v\in W_{sl_2}(\dot{e})$ be any nonzero vector, then $v$ is of
the form $v=\sum_{i,j}a_{i,j}f^ih^jw.$ Define
$$\mathrm{deg}_f(v)=\mathrm{max}\{i|a_{i,j}\neq0\}$$ and
$$W_{sl_2}^{k}(\dot{e})=\mathrm{span}\{v|v\in W_{sl_2}(\dot{e}),\mathrm{deg}_f(v)\leq k.\}$$

\begin{lemm}
Suppose $\dot{e}\neq0.$

$\mathrm{(i)}$ If $\bar{\Omega}\in U(sl_2)$ is of the form
 $\bar{\Omega}=f+\bar{\Omega}_1,$ where $\bar{\Omega}_1\in \mathbb{C}[h],$
such that $\bar{\Omega}w$ is a Whittaker vector, then
$\bar{\Omega}=f+\frac{h}{2\dot{e}}+\frac{h^2}{4\dot{e}}+c$ for some
$c\in\mathbb{C}.$

$\mathrm{(ii)}$ For any $a\in\mathbb{C},$ $i\in\mathbb{N},$ there
exists $\psi_0(x),\psi_1(x),\cdots,\psi_i(x)\in\mathbb{C}[x]$ such
that
$$f^iw=\frac{1}{(4\dot{e})^i}(\Omega-a)^iw+\psi_1(h)(\Omega-a)^{i-1}w+\cdots+\psi_{i-1}(h)(\Omega-a)w+\psi_i(h)w,$$
where $\Omega=4fe+2h+h^2$ is the Casimir element of $sl_2.$
\end{lemm}
\begin{proof}$\mathrm{(i)}$ Assume $\bar{\Omega}=f+\sum_{i=0}^na_ih^i, a_n\neq0$
and $\bar{\Omega}w$ is a Whittaker vector. By
$(e-\dot{e})\bar{\Omega}w=0$ and (2.6), we can easily get that
$\bar{\Omega}=f+\frac{h}{2\dot{e}}+\frac{h^2}{4\dot{e}}+c$ for some
$c\in\mathbb{C}.$

$\mathrm{(ii)}$ By $\mathrm{(i)}$ and the fact that $\Omega\in
Z(sl_2),$ the center of $U(sl_2),$ we have
\begin{align*} f^{i+1}w&=
f^{{i}}(f+\frac{h}{2\dot{e}}+\frac{h^2}{4\dot{e}})w(\mathrm{mod}W_{sl_2}^{i}(\dot{e}))
\\
&\equiv f^{{i}}(4fe+2h+h^2)\frac{1}{4\dot{e}}w
(\mathrm{mod}W_{sl_2}^{i}(\dot{e}))
\\
&\equiv f^{{i-1}}(4fe+2h+h^2)f\frac{1}{4\dot{e}}w
(\mathrm{mod}W_{sl_2}^{i}(\dot{e}))\\
&\equiv f^{{i-2}}(4fe+2h+h^2)^2\frac{1}{(4\dot{e})^2}w
(\mathrm{mod}W_{sl_2}^{i}(\dot{e}))\\
&\equiv \cdots\\
&\equiv (4fe+2h+h^2)^{i+1}\frac{1}{(4\dot{e})^{i+1}}w
(\mathrm{mod}W_{sl_2}^{i}(\dot{e})).
\end{align*}
By induction, $\mathrm{(ii)}$ holds.
\end{proof}

\begin{lemm}
For Whittaker $sl_2-$module $W_{sl_2}(\dot{e}),$ the following hold:

$\mathrm{(i)}$ If $\dot{e}=0,$  the set of Whittaker vectors of
$W_{sl_2}(0)$ is $\mathbb{C}[h]w.$

$\mathrm{(ii)}$ If $\dot{e}\neq0,$  the set of Whittaker vectors of
$W_{sl_2}(\dot{e})$ is $\mathbb{C}[\Omega]w.$
\end{lemm}
\begin{proof}
(i) Suppose $\dot{e}=0.$ By (2.6), any nonzero element of
$\mathbb{C}[h]w$ is a Whittaker vector. By (2.5) and (2.6),  we see
that $ev\neq0$ if $\mathrm{deg}_f(v)\neq0,$ which means $v$ is not a
Whittaker vector.

 (ii) Suppose $\dot{e}\neq0.$ On one hand, every nonzero vector of
$\mathbb{C}[\Omega]w$ is a Whittaker vector since $\Omega\in
Z(sl_2)$.

On the other hand, for any Whittaker vector
$v=\sum_{i=0}^na_i\psi_{i}(h)f^iw\in W_{sl_2}(\dot{e})$ with
$\mathrm{deg}_f(v)=n,$ by Lemma 4.4 (ii), we see that $v$ is of the
form
$$v=\sum_{i=0}^nb_i\varphi_i(h)(\Omega-a)^iw$$
for $\varphi_{x}\in\mathbb{C}[x],
\varphi_n(x)\neq0,b_i\in\mathbb{C},b_n\neq0.$ By (2.6), it is easy
to deduce that either $\varphi_i(x)=0$ or
$\mathrm{deg}(\varphi_i(x))=0.$
\end{proof}

\begin{prop}
For Whittaker $sl_2-$module $W_{sl_2}(\dot{e}),$ the following hold:

$\mathrm{(i)}$ If $\dot{e}=0,$ $V$ is a $sl_2-$submodule of
Whittaker module $W_{sl_2}(0),$  then $V$ is a maximal submodule if
and only if there exists a scalar $\alpha$ such that $V$ is
generated by $(h-\alpha)w.$

$\mathrm{(ii)}$ If $\dot{e}\neq0,$  for any
 $\dot{\Omega}\in\mathbb{C},$ the Whittaker module $W_{sl_2}(\dot{e})$ has the following filtration:
 $$W_{sl_2}(\dot{e})=W^0\supset W^1\supset W^2\supset\cdots \supset W^n\supset
 \cdots,$$ where $W^i$ is a Whittaker submodule defined by $W^i=U(sl_2)(\Omega-\dot{\Omega})^iw.$ More precisely,
  $M_{sl_2}(\dot{e},\dot{\Omega}):=W_{sl_2}(\dot{e})/W^1$
  is simple and $W^i/W^{i+1}\cong W_{sl_2}(\dot{e})/W^1$ for all
 $i\in\mathbb{N}.$
\end{prop}
\begin{proof} We can use the same arguments as in Lemma 3.2 and Theorem 3.8 (ii) to
complete the proof. We omit the details.
\end{proof}

\begin{rema}
It is follows from Proposition 4.6 (i) that if $\dot{e}=0,$ $V$ is a
maximal submodule of $W_{sl_2}(0)$ generated by $(h-\alpha)w,$ then
$W_{sl_2}(0)/V$ is a simple highest weight $sl_2$-module with
highest weight $\alpha$.
\end{rema}

\begin{theo}
Let $V$ be a simple Whittaker module of type $\phi:e \rightarrow
\dot{e}$.

$\mathrm{(i)}$ If $\dot{e}=0,$ then there exists
$\alpha\in\mathbb{C}$ such that $V$ is isomorphic to a simple
highest weight $sl_2$-module with highest weight $\alpha$.

$\mathrm{(ii)}$ If $\dot{e}\neq0,$ then $V\cong
M_{sl_2}(\dot{e},\dot{\Omega})$ for some $\dot{\Omega}\in
\mathbb{C}.$
\end{theo}
\begin{proof}
$\mathrm{(i)}$ follows from Proposition 4.6 (i) and Remark 4.7. The
proof of $\mathrm{(ii)}$ is similar as that of Theorem 3.8 (ii). We
omit the details.
\end{proof}

\begin{coro}
Let $V$ be a Whittaker $sl_2-$module of type $\phi:e \rightarrow
\dot{e}$.

$\mathrm{(i)}$ If $\dot{e}=0,$ $V$ is simple if and only if $V$ is a
highest weight module.

$\mathrm{(ii)}$ If $\dot{e}\neq0,$ then $V$ is simple if and only if
 the Casimir element of $sl_2$ acts as a scalar.
\end{coro}

\begin{theo}
Let $W(\dot{e},\dot{p},\dot{z})$ be a Whittaker
$\mathfrak{S}-$module of type $\phi: e \rightarrow \dot{e}, p
\rightarrow \dot{p}$ of nonzero level $\dot{z}.$ Then
$W(\dot{e},\dot{p},\dot{z})$ is simple if and only if there exists a
simple Whittaker $\mathfrak{H}-$module
$M_{\mathfrak{H}}(\dot{p},\dot{z})$ and a simple Whittaker
$\mathfrak{sl_2}-$module
$M_{\mathfrak{sl_2}}(\dot{e}-\frac{1}{2\dot{z}}\dot{p}^2,\dot{\Omega})$
such that $W(\dot{e},\dot{p},\dot{z})\cong
M_{\mathfrak{H}}^{\mathfrak{S}}(\dot{p},\dot{z})\otimes
M_{sl_2}^{\mathfrak{S}}(\dot{e}-\frac{1}{2\dot{z}}\dot{p}^2,\dot{\Omega}).$
\end{theo}
\begin{proof} Suppose first that $W(\dot{e},\dot{p},\dot{z})$ is
simple Whittaker module of nonzero level.
 Denote by
$U(\mathfrak{H})_{(z)}$ the localization of $U(\mathfrak{H})$ with
respect to the multiplicative subset $\{z^{i}|i\in\mathbb{Z}_+\}$.
By Theorem 1 (i) in [26] (here $z$ is opposite to that of in [26]),
 we see that there is a unique algebra
homomorphism $\Phi:U(\mathfrak{S})\rightarrow U(\mathfrak{H})_{(z)}$
which is the identity on $U(\mathfrak{H})$ such that
$$\Phi(e)=\frac{1}{2z}p^2,\ \Phi(f)=-\frac{1}{2z}q^2,\
\Phi(h)=-\frac{1}{z}qp-\frac{1}{2}.$$ Then
$\mathrm{Res}_{\mathfrak{H}}^{\mathfrak{S}}W(\dot{e},\dot{p},\dot{z}),$
the restrictive module on $\mathfrak{H}$, contains a simple
Whittaker submodule $M_{\mathfrak{H}}(\dot{p},\dot{z}),$ which can
be viewed as a $\mathfrak{S}-$Whittaker module such that
$$ev=\Phi(e)v,fv=\Phi(f)v,hv=\Phi(h)v.$$
 We denote this new $\mathfrak{S}$-module by $W_{\mathfrak{H}}^{\mathfrak{S}}(\dot{p},\dot{z}).$
Then, by Theorem 3 (ii) in [26], we have a simple $sl_2-$Whittaker
module $M_{sl_2}(\dot{e}-\frac{1}{2\dot{z}}\dot{p}^2,\dot{\Omega})$
such that $W\cong
M_{\mathfrak{H}}^{\mathfrak{S}}(\dot{p},\dot{z})\otimes
M_{sl_2}^{\mathfrak{S}}(\dot{e}-\frac{1}{2\dot{z}}\dot{p}^2,\dot{\Omega}),$
where the simple $\mathfrak{S}-$module
$M_{sl_2}^{\mathfrak{S}}(\dot{e}-\frac{1}{2\dot{z}}\dot{p}^2,\dot{\Omega})$
is from simple $sl_2-$module
$M_{sl_2}(\dot{e}-\frac{1}{2\dot{z}}\dot{p}^2,\dot{\Omega})$ by
setting $\mathfrak{H}v=0.$

Conversely, if $M_{\mathfrak{H}}(\dot{p},\dot{z})$ is a simple
Whittaker $\mathfrak{H}$-module of nonzero level $\dot{z}$ and
$M_{sl_2}(\dot{e}-\frac{1}{2\dot{z}}\dot{p}^2,\dot{\Omega})$ is a
simple Whittaker $sl_2-$module. Then we can view
$M_{\mathfrak{H}}(\dot{p},\dot{z})$ as a $\mathfrak{S}-$ module,
denote by $M_{\mathfrak{H}}^{\mathfrak{S}}(\dot{p},\dot{z}),$ by
defining $ev=\Phi(e)v,fv=\Phi(f)v,hv=\Phi(h)v.$ We also can view
$M_{sl_2}(\dot{e}-\frac{1}{2\dot{z}}\dot{p}^2,\dot{\Omega})$  as a
$\mathfrak{S}-$module by defining
$\mathfrak{H}.M_{sl_2}(\dot{e}-\frac{1}{2\dot{z}}\dot{p}^2,\dot{\Omega})=0.$
then $M_{\mathfrak{H}}^{\mathfrak{S}}(\dot{p},\dot{z})\otimes
M_{sl_2}^{\mathfrak{S}}(\dot{e}-\frac{1}{2\dot{z}}\dot{p}^2)$ is a
Whittaker $\mathfrak{S}-$module of type $\phi: e\rightarrow\dot{e},
p\rightarrow\dot{p}$ of nonzero level $\dot{z}.$  By Theorem 3 (i)
in [26], we see that
$M_{\mathfrak{H}}^{\mathfrak{S}}(\dot{p},\dot{z})\otimes
M_{sl_2}^{\mathfrak{S}}(\dot{e}-\frac{1}{2\dot{z}}\dot{p}^2)$ is
simple. This completes the proof of Theorem 3.9.
\end{proof}

\begin{rema}
By Proposition 4.2, Theorem 4.8 and Theorem 4.10, the Whittaker
modules of nonzero level for Schr\"{o}dinger algebra are classified.
\end{rema}

\end{document}